\documentclass[a4paper]{amsart}

\usepackage{amsthm,amsmath,amsfonts,amssymb,amsrefs,mathtools,comment}

\theoremstyle{plain}
\newtheorem{thm}{Theorem}[section]
\theoremstyle{definition}

\theoremstyle{plain}
\newtheorem{lem}{Lemma}[section]
\theoremstyle{plain}
\newtheorem*{nclem}{The $C^r$ Closing Lemma}
\theoremstyle{plain}
\newtheorem*{conlem}{The $C^r$ Connecting Lemma}
\theoremstyle{plain}
\newtheorem*{hypthm}{Hyperbolicity Theorem}
\theoremstyle{plain}

\theoremstyle{plain}
\newtheorem{cor}{Corollary}[section]
\theoremstyle{plain}

\theoremstyle{plain}
\newtheorem{thm1}{Theorem}
\theoremstyle{plain}
\newtheorem{defis}{Definition}[section]
\theoremstyle{plain}

\newtheorem*{htl}{Hyperbolic tendency lemma}
\theoremstyle{plain}

\newcommand\SVF{\mathfrak{X}(M)}

\newcommand\NW{\Omega}
\newcommand\NB{\mathcal{U}}
\newcommand\LIE{\mathcal{L}}
\newcommand\CD{\nabla}
\newcommand\oY{\overline{Y}}
\newcommand\oX{\overline{X}}
\newcommand\orho{\overline{\rho}}
\newcommand\mT{\mathcal{T}}
\newcommand\mP{\mathcal{P}}
\newcommand\mI{\mathcal{I}}
\newcommand\mO{\mathcal{O}}

\DeclareMathOperator{\dis}{dist}

\numberwithin{equation}{section}

\begin{document}

\title{a proof of the $C^r$ closing lemma and stability conjecture}
\author{Chang Gao}
\begin{abstract}
It has long been conjectured that generic dynamical systems has finite periodic orbits, ever since the time of Poincar\'e. In this article, a perturbation method is proposed for the $C^r$ closing of periodic orbits. This method is applicable to both time-varying and time-invariant vector fields. 
\end{abstract}
\maketitle

\section{Introduction}
Let $\SVF:M\to TM$ be the set of smooth vector fields endowed with $C^r$ topology on a locally compact Riemannian manifold $M$, as originally defined in \cite{AnPo1937}: a flow $\phi'_0=X\in\SVF$ is $C^r$ structurally stable if there is a neighborhood $\NB(X)\subset\SVF$ and a homeomorphism for each $\psi'_0=Y\in\mathcal{U}(X)$ so that $\phi_t\circ h=h\circ\psi_t$.\par
The genericity of structurally stable systems in the set of all differential equations on two dimensional closed manifolds has been established. In higher dimensions, counter examples were constructed \cite{Sm1966}. Yet, the stability conjecture serves as a link between the global long-time behaviour and local hyperbolicity of the nonwandering set of a differential system. It states that \cite{PaSm1970}: the necessary and sufficient condition for a flow $\phi_t$ to be $C^r$ structurally stable is\par
 \romannumeral 1) $\Omega(\phi_t)=\overline{P(\phi_t)}$\par
 \romannumeral 2) $\Omega(\phi_t)$ is hyperbolic\par
 \romannumeral 3) $W^s(x)$ and $W^u(y)$ intersect transversally for all $x,y\in\Omega(\phi_t)$.\par
 Systems which satisfy condition \romannumeral 1) and \romannumeral 2) are called Axiom A systems. For example, evidence shows that the strange attractor in Lorenz system contains an Axiom A subset \cite{Tu2002}. The basic idea behind constructing geometric chaotic system is that when invariant manifolds of periodic orbits accumulate on other invariant manifolds and form infinite transversal intersections for the Poincar\'e map, iterations of this accumulation procedure yield very complicated dynamical systems. This is the idea behind Smale's horseshoe map, for a thorough introduction on the geometric theory of Axiom A systems see \cite{Sm1967}. \par 
The sufficiency of the stability conjecture was proved in \cite{Ro1971}\cite{Rs1976} for $r\geq1$, necessity was solved in the $2$-dimensional case \cite{Pe1962}. In higher dimensions than $2$ it was reduced to prove that $C^r$ structural stability implies conditions \romannumeral 1),\romannumeral 2) \cite{Rs1973}. Then the conjecture was proved for $r=1$ in \cite{Ma1987}\cite{Ha1997} for diffeomorphisms and flows respectively.  In this paper it is assumed that all the eigenspace of $\Omega(\phi_t)$ is contained in $M$, otherwise they do not contribute to the global topological structure of $\phi_t$ in $M$, as this example shows \cite{LaPa1986}. We also make the following assumptions, that $M$ is a locally compact smooth manifold with a positive definite norm on the tangent bundle, and that the nonwandering set $\Omega(\phi_t)$ is bounded. This is the case generalized from previously cited results on closed manifolds. The main result states
\begin{thm1}\label{sc}
If a flow $\phi_t$ is $C^r$ structurally stable, then it satisfies condition \romannumeral 1) and \romannumeral 2).
\end{thm1} 

 Before giving the main theorems and proofs, let's review some definitions. One could find these definitions in expository text such as \cite{Gu1983}. \par
$P(\phi_t)$ denotes the set of periodic points $\{x\in M|\phi_T(x)=x\}$ where $T$ is the period of $x$. $\overline{P(\phi_t)}$ the closure of $P(\phi_t)$ in $M$.\par
 The nonwandering set $\Omega(\phi_t)$ is the set of $x\in\Omega(\phi_t)$ such that for every neighborhood $U$ of $x$ and $T>0$, $\phi_t(U)\cap U\neq\emptyset$ for some $t>T$. This definition differs slightly from the nonwandering points of a diffeomorphism, but doesn't affect the results extended from those of flows.\par
 A point $x\in\Omega(\phi_t)$ is hyperbolic if its tangent space has an invariant hyperbolic splitting, i.e. $TM_x=E^s\oplus E^u\oplus X$ and $d\phi_t|E^s\subset E^s$ such that 
$$\|d\phi_t|E^s(x)\|\leq C\lambda^t\quad t\geq0$$
$$\|d\phi_{-t}|E^u(x)\|\leq C\lambda^t\quad t\geq0$$
for some constants $C>0$, $0<\lambda<1$. 
 $$W^s(x)=\{y\in M|\lim_{t\to\infty}\|\phi_t(x)-\phi_t(y)\|=0\}$$
 $$W^u(x)=\{y\in M|\lim_{t\to\infty}\|\phi_{-t}(x)-\phi_{-t}(y)\|=0\}$$
are the stable and unstable manifolds of a point $x$, respectively. \par
Two fields $X$, $Y\in\SVF$ are $C^r$ close if for any $\alpha>0$, their Lie derivatives satisfy $\max\{\|\LIE^q_{h_q}X(x)-\LIE^q_{h_q}Y(x)\|_m|\|h_q\|_m=1, q=0,1,\dots,r\}\leq\alpha$, where $\LIE^q_{h_q}=\LIE_{h_q}\LIE_{h_{q-1}}\dots\LIE_{h_1}$, $h_q(x)\in T^{q+1}M=TM$, $\|*\|_m$ denotes the norm defined on the tangent space $TM$. In Euclidean space $\mathbb{R}$ this definition is equivalent to $\|D^r\|$. 

\section{Main Theorems and Outline of the Proof}\label{out}
A key theorem to the stability conjecture is the $C^r$ closing lemma, the proof of the lemma itself for $r\geq2$ was considered a very difficult problem that it was listed in \cite{Sm1999}, see \cite{An2012} for a general review. This problem was proved in \cite{Pu1967}\cite{Pu1967im} for $C^1$ small perturbations, there is also a proof of $C^\infty$ closing lemma for 3-dimensional Hamiltonian flows \cite{As2016}. \par
Major obstruction in extending the result from $C^1$ to $C^r$ case for $r\geq2$ lies with the fact that the $C^r$-size of the perturbation grows reciprocally to $\epsilon^r$ as $r$ increases, where $\epsilon$ is the size of perturbed area, or size of the flow box. Another obstruction which was solved in the $C^1$ case is that perturbing the vector field could destroy previous return points of the orbit. To solve these problems we propose making the perturbation within a 'long flow box' (definition \ref{lfb}), surrounding \textit{a priori} constructed periodic orbit. The perturbation can be constructed by integrating the linearization of the original vector field $dX$ along geodesics to points on the cross-section, where the initial vector lies on the  \textit{a priori} constructed periodic orbit. This turns out to be equivalent to parallel transporting the vector, see the expression of the perturbation in the theorem below
\begin{nclem}
If a flow $\phi_t$ induced by a vector field $X\in\SVF$ has a nonperiodic nonwandering point $\omega$, then for any neighborhood $\NB(X)\subset\SVF$, there is a flow $\psi_t$ induced by vector field $Y\in\NB(X)$, so that $p$ is periodic for $\psi_t$ and is arbitrarily close to $\omega$. The time-dependent perturbation between $X$ and $Y$ satisfies
\[
Y(x,t)-X(x,t)=\rho(\dis(\pi_J(x),x)+L|t-t_x|)(\oY(x)-\oX(x))
\]
where $\oY$ and $\oX$ are parallel transports along geodesics normal to $\psi_t(p)$.
\end{nclem}
Notice that the above perturbation is nonautonomous, which was set up to solve the nonautonomous closing lemma. But the same idea could be applied to the autonomous case, by superpositioning the parallel transport of different sections of the periodic orbit, this perturbation takes the form (\ref{hom20}) with the time dependent term removed. I refer the reader to (\ref{clm31}) and theorem \ref{clm1} for detailed proof. 

 Using similar construction of $C^r$ small perturbation as in the closing lemma, one obtains the $C^r$ connecting lemma (theorem \ref{con1}), which was proved in \cite{Ha1997} for $r=1$, see also \cite{Wen1995} for its implication on the $C^1$ stability conjecture.
\begin{conlem}
If $X\in\SVF$ is $C^r$ structurally stable, then no almost periodic point is close to equilibrium point.
\end{conlem}
See definition \ref{aml} for almost periodic point. Since $\phi_t$ is structurally stable hence homeomorphic to $\psi_t$, we prove that periodic points are dense in the set of nonwandering points $\NW(X)$ in corollary \ref{clm2}, i.e., necessity of condition \romannumeral 1) of the stability conjecture. \par
A key lemma to the necessity of condition \romannumeral 2) is proved in lemma \ref{hyp0}, this lemma proves the hyperbolic tendency of $C^r$ structurally stable periodic points. The $C^1$ case was proved by \cite{Fr1971}. The obstruction in this lemma mentioned in the introduction of \cite{Ma1982} I believe to be that perturbing the eigenvalues of vector field could change the return time of its flow. We overcome this problem by making a cross-section preserving homeomorphic perturbation, so it preserves the return time function $T(p,\psi)$ as well. The lemma states as follows
\begin{htl}
If $X$ is $C^r$ structurally stable, $P(X)$ is hyperbolic, and its eigenvalues cannot tend to absolute value $1$.
\end{htl}
See lemma \ref{hyp0}. Intuitively the eigenvalues hence hyperbolicity of nonwandering set can be approximated by eigenvalues of dense periodic points. An obstruction to this approach is that the periodic orbits wander off the nonwandering orbit when the period tends to infinity. Suppose that the sequence of minimum periods is unbounded, i.e. $\lim_{i\to\infty}T(p_i)\to\infty$. We proved in the $C^r$ connecting lemma that if $X$ is structurally stable, then $X(x)\neq0$ for all $x\in\Omega(X)$. This leads to the case that $\phi_t(p_i)$ is dense somewhere in $\Omega(X)$. Since each $p_i$ is hyperbolic, there is a neighborhood of $p_i$ where the first return time $T(x)$ exist and is smooth. Since $\Omega(X)$ is bounded, by lemma \ref{nonauto} there are times $0\leq T_0<T_1<\infty$ where $\phi_{T_1(x)}(p_i)$ return arbitrarily $C^r$ close to $\phi_{T_0(x)}(p_i)$. Suppose $d\phi_{T_1-T_0}$ has eigenvalue arbitrarily close to $1$, the $C^r$ closing lemma could be applied to create periodic points with eigenvalue arbitrarily close to $1$. Another perturbation would create nonhyperbolic periodic points, contradicting structural stability. Thus we have proved that every periodic point of a structurally stable vector field with bounded nonwandering set has finite hyperbolic sub-period. Finally, we approximate the eigenvalues of nonwandering points by eigenvalues of periodic points.
\begin{hypthm}\label{hy}
If $X$ is $C^r$ structurally stable, $\overline{P(X)}$ is hyperbolic.
\end{hypthm} 
See theorem \ref{hyp1} for details. By density of periodic points and hyperbolicity of $\NW(X)=\overline{P(X)}$, theorem \ref{sc} is proved.

\section{Proof of the $C^r$ Closing Lemma}\label{pclm}
This section contains a proof of the $C^r$ closing lemma (theorem \ref{clm1}), necessity of condition \romannumeral 1) is proved in corollary \ref{clm2}.
Let $\omega\in\NW(\phi_t)$ be a nonwandering point of the flow $\phi_t$ induced by $X\in\SVF$, $M$ a locally compact smooth Riemannian manifold with a given connection $\CD$ compatible with metric, denote $\dis(x,y)$ the distance between $x$ and $y$ on $M$. 
\begin{defis}\label{aml}
A point $x_0\in \phi_t(U(\omega))\cap U(\omega)$ is almost periodic for $\phi_t$ if for any neighborhood $U$ of $\omega$, there is $T>0$ such that $U(\omega)\supset x_0'=\phi_{T}(x_0)$.
\end{defis}
If $x_0$, $x_1$, \dots , $x_{n+1}=x_0$ $\in U(\omega)$ is a sequence of almost periodic points associated to the nonwandering point $\omega$, denote their first return points by $x_0'=\phi_{t_1}(x_0,0)$, $x_1'=\phi_{t_2-t_1}(x_1,t_1)$, \dots , $x_n'=\phi_{t_{n+1}-t_n}(x_n,t_n)$ $\in U(\omega)$, $T=t_{n+1}$ the period of $x_0$. In the remaining part of this paper, $\phi_t(x)$ will be short for $\phi_t(x,t_0)$ or $\phi_{t\bmod T}(x)$ whenever it appears.\par
The following lemma ensures that in case of nonautonomous systems, there are almost periodic points in $U(\omega)$ where the vector field returns $C^r$ close. 
\begin{lem}\label{nonauto}
For every almost periodic point $x_0\in U(\omega)$ and any $\alpha>0$, there exists a sequence of almost periodic points $x_i\in U(\omega)$, $i=0,1,2,\dots ,n$ and a geodesic $\gamma_{s_i}(x_i')=\gamma_{i'}=x_0$ for each $i$, such that:
\begin{equation}\label{na1}
\|\LIE_{h_q}^q\oX(\gamma_{i+1},t_{i+1})-\LIE_{h_q}^q\oX(\gamma_{i'},t_{i+1})\|<bL^r\alpha
\end{equation}
for all $q=0,1,\dots,r$ and $x_n=x_0$, where $b$ and $L$ are constants depending only on $X$, $\oX(\gamma_{i+1},t_{i+1})$ is the solution to the differential equation of parallel transport: 
\begin{equation}\label{na2}\CD_{\gamma'_{i+1}}\oX|_{\gamma_{i+1}}=0\quad\mathrm{and}\quad \oX(x_{i+1},t_{i+1})=X(x_{i+1},t_{i+1}).
\end{equation}
\end{lem}
\begin{proof}
First consider the case $i<n-1$, let $x^j$ be a local coordinate system on $M$, equation (\ref{na2}) writes:
$$\frac{d}{ds}\overline{X^j}(\gamma_{i+1},t_{i+1})+\gamma'^k\overline{X^l}(\gamma_{i+1},t_{i+1})\Gamma^j_{kl}=0$$
It is standard fact of differential equation that the above parallel transport equation and geodesic equation has smooth solutions locally: $X^j(\gamma_{i+1},t_{i+1},i+1)$ and $\gamma_s(x_{i+1})$ respectively, then there is a Lipschitz constant $L$ such that
\begin{equation}\label{na3}\begin{aligned}
|\overline{X^j}(\gamma_{i+1},t_{i+1})-X^j(x_{i+1},t_{i+1})|&<L\dis(\gamma_{i+1},x_{i+1})\\
|\overline{X^j}(\gamma_{i'},t_{i+1})-X^j(x_{i'},t_{i+1})|&<L\dis(\gamma_{i'},x'_{i})
\end{aligned}\end{equation}
 furthermore, since $h_q\in\SVF$
\begin{multline}\label{na4}|dx^j[h_q,[h_{q-1},[\dots,[h_1,X(x_{i+1})]\dots]-dx^j[h_q,[h_{q-1},[\dots,[h_1,X(x'_{i})]\dots]|
\\
<L^q\dis(x_{i+1},x'_i)
\end{multline}
for all $q=0,1,\dots,r$. Combine (\ref{na3}) and (\ref{na4}) one obtains
\begin{equation}\label{na5}\begin{aligned}
&\|\LIE_{h_q}^q\oX(\gamma_{i+1},t_{i+1})-\LIE_{h_q}^q\oX(\gamma_{i'},t_{i+1})\|\\
&<\|\sum_j(|dx^j[h_q,[\dots,[h_1,\oX(\gamma_{i+1})]\dots]-dx^j[h_q,[\dots,[h_1,X(x_{i+1})]\dots]|\frac{\partial}{\partial x^j}\\
&+|dx^j[h_q,[\dots,[h_1,\oX(\gamma_{i'})]\dots]-dx^j[h_q,[\dots,[h_1,X(x'_{i})]\dots]|\frac{\partial}{\partial x^j}\\
&+|dx^j[h_q,[\dots,[h_1,X(x_{i+1})]\dots]-dx^j[h_q,[\dots,[h_1,X(x'_{i})]\dots]|
\frac{\partial}{\partial x^j})\|\\
&<\sum_j(L^q\dis(\gamma_{i+1},x_{i+1})\|\frac{\partial}{\partial x^j}\|+L^q\dis(\gamma_{i'},x'_{i})\|\frac{\partial}{\partial x^j}\|
+L^q\dis(x_{i+1},x'_{i})\|\frac{\partial}{\partial x^j}\|)
\end{aligned}\end{equation}
since $\omega$ is nonwandering and $x_{i+1},x'_i,\gamma_{i+1}=\gamma_{i'}=x_0\in U(\omega)$, the three distance in the last inequality of (\ref{na5}) $<\alpha$:
\[
\|\LIE_{h_q}^q\oX(\gamma_{i+1},t_{i+1})-\LIE_{h_q}^q\oX(\gamma_{i'},t_{i+1})\|<bL^q\alpha
\]
for all $i<n-1$, where $b=3\sum_j\|\frac{\partial}{\partial x^j}\|$ is a positive constant.\par
Now to prove equation (\ref{na1}) for $i=n-1$.  Since $M$ is locally compact and $X\in\SVF$, the set of  vectors in the tangent bundle $v\in TM\oplus TTM\oplus\dots\oplus T^rM|U(\omega)=A$ is bounded, i.e. $\|v\|_{\oplus}\coloneqq\sum_k\|v|T^kM\|<K$ for some positive constant $K$. Therefore for any $\alpha>0$, there is a finite collection of open balls $B_{bL^r\alpha}(v_j)$ centered at $v_j$ with radius $bL^r\alpha$, $v_j\in A$, $j=0,1,\dots,N_j$, such that 
\begin{equation}\label{na6}A\subset\bigcup_{j=0}^{N_j}B_{bL^r\alpha}(v_j)
\end{equation}
 If one choose a sequence of $v_j$, so that 
$$v_j|T^{q+1}M=\LIE_{h_q}^q\oX(\gamma_{j+1},t_{j+1})$$
 then by (\ref{na6}), $v_n\cap A\neq\emptyset$ for any $n>N_j$, i.e. $v_n\in B_{bL^r\alpha}(v_0)$, therefore
$$\|\LIE_{h_q}^qX(x_0,t_{0})-\LIE_{h_q}^q\oX(\gamma_{n'},t_{n})\|=\|v_0|T^{q+1}M-v_n|T^{q+1}M\|<bL^r\alpha$$
for all $0\leq q\leq r$.
\end{proof}
For the sake of simplicity, from now on we assume that $N_j=1$ and $X(x)$ is short for $X(x,t)$. 
The concept of flow box comes from the theory of exterior differential forms, and was used in the proof of $C^1$ closing lemma \cite{Pu1967} to define local coordinates and make perturbations within. The same idea is modified in the next definition so that the flow box extends along the trajectory.
\begin{defis}\label{lfb}
A long flow box $F(J)$ of a trajectory $J(x_0)=\{\phi_t(x_0)|0\leq t\leq T\}$ is the set of points in the cross-section perpendicular to $J$ with distance less than $\epsilon$ to the $J$.
\end{defis}
Let $\Sigma(x_0)$ be the cross-section normal to $J(x_0)$. By choosing $\epsilon$ small enough, one can establish a long flow box for $J$ so that the projection map $\pi_J|F(J)$ is unique, the next theorem is a detailed restatement of the nonautonomous closing lemma.
\begin{thm}\label{clm1}
If a field $X\in\SVF$ has nontrivial nonwandering point $\omega$, there is a field $Y\in\SVF$ $C^r$-close to $X$ such that
\begin{equation}\label{clm000}\begin{aligned}
\|\LIE^q_{h_q}X|F(J)-\LIE^q_{h_q}Y|F(J)\|
&<\frac{\rho_0}{\epsilon^q}b^2L^{2r}(L^r+H^r)(\epsilon L^r+2)\alpha\\
X|M\setminus F(J)&=Y|M\setminus F(J)
\end{aligned}\end{equation}
for all $q=0,1,\dots,r$, and $J(x_0)$ is a periodic orbit for $Y$, $x_0\in U(\omega)$:
$$\CD^q_{\psi'_t}\psi_0(x_0)=\CD^q_{\psi'_t}\psi_T(x_0)$$
\end{thm}
\begin{proof}
The proof is divided into 3 parts. \par
\textbf{The periodic orbit $J(x_0)$} Since $x_0\in U(\omega)$, there is a coordinate chart $x^j$ on $U(\phi_t(x_0))\subset M$ such that a interpolation $f_t(x_0)\subset U(\phi_t(x_0))$ could be constructed \cite{Sp1960}, which satisfies
\begin{equation}\label{clm03}\begin{aligned}
&\CD^q_{f'_t}f_0(x_0)=\CD^q_{f'_t}f_T(x_0)\\
&|\frac{d^q}{dt^q}f^j_t(x_0)-\frac{d^q}{dt^q}\phi^j_t(x_0)|<H^q|\frac{d^q}{dt^q}\phi^j_0(x_0)-\frac{d^q}{dt^q}\phi^j_T(x_0)|\quad\mathrm{for}\quad0\leq t\leq T\\
&<bH^rL^r\alpha
\end{aligned}\end{equation}
where $H$ is a positive constant depending on the type of interpolation (Hermite interpolation in this case), the last inequality follows from inequality (\ref{na1}) by setting $h_q=X$ and $n=0$. Put equation (\ref{clm03}) into the estimation of $C^r$-size of this interpolation yields
\begin{equation}\label{clm04}\begin{aligned}
&\|\overbrace{f'_t\dots f'_t}^q X(f_t)-\overbrace{f'_t\dots f'_t}^q f'_t(f_t)\|\\
&<|dx^j\overbrace{f'_t\dots f'_t}^q X(f_t)-dx^j\overbrace{f'_t\dots f'_t}^q X(\phi_t)|\|\frac{\partial}{\partial x^j}\|\\
&+|dx^j\overbrace{f'_t\dots f'_t}^q X(\phi_t)-dx^j\overbrace{f'_t\dots f'_t}^q f'_t( f_t))|\|\frac{\partial}{\partial x^j}\|\\
&<\sum_j(L^q|\frac{d^q}{dt^q}f^j_t(x_0)-\frac{d^q}{dt^q}\phi^j_t(x_0)|\|\frac{\partial}{\partial x^j}\|\\
&+H^q|\frac{d^q}{dt^q}f^j_t(x_0)-\frac{d^q}{dt^q}\phi^j_t(x_0)|\|\frac{\partial}{\partial x^j}\|)\\
&<b^2(L^r+H^r)L^r\alpha
\end{aligned}\end{equation}
for all $q=0,1,\dots,r$. One obtains a periodic orbit $J(x_0)=\{f_t(x_0)|t\in\mathbb{R}\}$.

\textbf{uniqueness of $\pi_J$.} For simplicity, assume $L_{min}>0$ is a lower bound for $\|X|J\|$, as will be shown in theorem \ref{con1} a different long flow box needs to be constructed if $\|X\|\to0$ without bound. Let $s(t)$ be a reparameterization by arc length for $f_t$ such that 
\begin{equation}\label{clm11}\begin{aligned}
\|\CD_{f'_{s(t)}}f_{s(t)}\|&=|s'(t)|\|\CD_{f'_s}f_s\|=1\\
s'(t)&\geq0
\end{aligned}\end{equation}
combine equation (\ref{clm04}) and (\ref{clm11}) one obtains an estimation for the curvature of $J$
\begin{equation}\begin{aligned}
\|\CD^2_{f'_{s(t)}}f_{s(t)}\|&=\|\CD_{f'_{s(t)}}(s'(t)\CD_{f'_s}f_s)\|=\|s'(t)\frac{d}{ds}s'(t)\CD_{f'_s}f_s+\CD_{f'_{s(t)}}\CD_{f'_s}f_s\|\\
&<|s'(t)| \|\CD_{f'_s}f_s\| \|\frac{d}{ds}\frac{1}{\|\CD_{f'_s}f_s\|}\| +|s'(t)|\|\CD_{f'_s}f'_s\|\\
&<\frac{\sum_j|\frac{d^2}{ds^2}f^j|\|\frac{\partial}{\partial x^j}\|}{\|X|J\|^2-\|X|J-Y|J\|^2}+\frac{1}{\|\CD_{f'_s}f'_s\|}\sum_{i,j,k}|\frac{d^2}{ds^2}f^j_s+f'^i_sf'^j_s\Gamma^j_{ij}|\|\frac{\partial}{\partial x^j}\|\\
&<\frac{L+b^2(L^r+H^r)L^r\alpha}{L_{min}^2-(b^2(L^r+H^r)L^r\alpha)^2}+\frac{bL^2+b^2(L^r+H^r)L^r\alpha}{L_{min}}
\end{aligned}\end{equation}
since $b$, $L, H$ are bounded constants, $L_{min}>0$ and $\alpha\to0$, an estimation of the radius of curvature $Rc$ of $J$ yields:
\[Rc>\frac{1}{\sup\|\CD^2_{f'_{s(t)}}f'_{s(t)}\|}>0\]
Therefore, a long flow box $F(J)$ with radius $0<\epsilon<\inf Rc$ has unique projection $\pi_J|F(J)$ locally in the $\epsilon$-neighborhood of $\pi_J(x)$. For $\pi_J$ to be nonunique, it is necessary that $f_t$ leave this $\epsilon$-neighborhood and then return to it, since $\|f'_t\|<L$, the minimum time of return $t_{min}>\frac{\epsilon}{L}$, i.e. if the nonautonoums perturbation vanishes before $\frac{\epsilon}{L}$, the projection is unique. \par

 \textbf{$C^r$ size between $X$ and $Y$} Let $Y$ be a smooth extension of $Y|f_t:=f'_t|f_t$ to $M$, which satisfies
\begin{equation}\label{clm31}\begin{aligned}
Y(x,t):=X(x,t)+\rho(\dis(\pi_J(x),x)+L|t-t_x|)(\oY(x)-\oX(x))
\end{aligned}\end{equation}
where $f_{t_x}(x_0)=\pi_J(x)$, $\rho(d)$ is a positive scalar function such that
\begin{equation}\label{clm{32}}\begin{aligned}
\rho(0)=1\quad\rho^{(q)}(d)=0\quad d\geq\epsilon\\
0<|\rho^{(q)}(d)|<\frac{\rho_0}{\epsilon^q}\quad0\leq d\leq\epsilon
\end{aligned}\end{equation}
for all $q=0,1,\dots,r$. $\rho_0$ is a positive bounded constant. $\oY$ and $\oX$ are parallel transports along a geodesic $\gamma_\tau(\pi_J(x))=x$ between $x$ and $\pi_J(x)$, they satisfy
\begin{equation}\label{clm33}\begin{aligned}
\CD_{\gamma'_\tau(\pi_J(x))} \oX|\gamma_\tau(\pi_J(x))=0\quad\oX(\pi_J(x))=X(\pi_J(x))\\
\CD_{\gamma'_\tau(\pi_J(x))} \oY|\gamma_\tau(\pi_J(x))=0\quad\oY(\pi_J(x))=Y(\pi_J(x))
\end{aligned}\end{equation}
Since the connection $\CD$ is compatible with metric, any vector in $TM$ under parallel transport preserves inner product, i.e.
\begin{equation}\label{clm34}
\|\oX|\gamma_\tau-\oY|\gamma_\tau\|=\|X(\pi_J(x))-Y(\pi_J(x))\|
\end{equation}
and since $\pi_J$ is unique in $F(J)$, there is a splitting such that
\[
TM_x=\Sigma(\pi_J(x))\oplus\oY(x)\quad x\in F(J)
\]
where $\langle v,\oY\rangle=\langle \gamma'_0,Y\rangle=0$ for any $v\in\Sigma$. First let's restrict attention to $h|\oY$, consider the parallel transport of higher order derivatives, by (\ref{clm33})
\[\begin{aligned}
&\|\CD_{\gamma'}\circ\CD^q_{h|\oY}(\oX-\oY)-\CD_{\gamma'}\circ\overline{\CD^q_{h|Y}(X-Y)}\|\\
&=\|\CD_{\gamma'}(h|\oY\CD^{q-1}_{h|\oY}(\oX-\oY)+h|\oY\cdot\CD^{q-1}_{h|\oY}(\oX-\oY)\cdot\Gamma)\|\\
&=\|0+\oY\cdot\CD^{q-1}_{h|\oY}(\oX-\oY)\cdot\CD_{\gamma'}\Gamma\|\\
&<b^2L^{q+r}(L^r+H^r)\alpha
\end{aligned}\]
where $\CD_{\gamma'}\circ\overline{\CD^q_{h|Y}(X-Y)}=0$, the last inequality follows inequality (\ref{clm04}) equation (\ref{clm34}) and induction on $q$. Integrate the above inequality one obtains
\[\begin{aligned}
&\|\CD^q_{h|\oY}(\oX-\oY)-\overline{\CD^q_{h|Y}(X-Y)}\|\\
&<\int^\epsilon_{0}\|\CD_{\gamma'}\circ\CD^q_{h|\oY}(\oX-\oY)-\CD_{\gamma'}\circ\overline{\CD^q_{h|Y}(X-Y)}\|d\tau\\
&<b^2L^{q+r}(L^r+H^r)\epsilon\alpha
\end{aligned}\]
since parallel transport preserves norm:
\[\|\overline{\CD^q_{h|Y}(X-Y)}\|=\|\CD^q_{h|Y}(X-Y)\|
\]
by inequality (\ref{clm04})
\[\begin{aligned}
\|\CD^q_{h|\oY}(\oX-\oY)\|
&<\|\CD^q_{h|\oY}(\oX-\oY)-\overline{\CD^q_{h|Y}(X-Y)}\|
+\|\overline{\CD^q_{h|Y}(X-Y)}\|\\
&<b^2L^{q+r}(L^r+H^r)\epsilon\alpha+b^2L^{r}(L^r+H^r)\alpha
\end{aligned}\]
one obtains an estimation
\begin{equation}\label{clm315}\begin{aligned}
\|\LIE^q_{h|\oY}(\oX-\oY)\|
&=\|\sum^q_{l=0}\CD^l_{h|\oY}(\oX-\oY)\overbrace{\oY\dots\oY}^l-\LIE^{q-1}_{h|\oY}(\oX-\oY)\|\\
&<\sum^q_{m=0}\sum^m_{l=0}\|\CD^l_{h|\oY}(\oX-\oY)\overbrace{\oY\dots\oY}^l\|\\
&<b^2L^{q+2r}(L^r+H^r)\epsilon\alpha+b^2L^{2r}(L^r+H^r)\alpha
\end{aligned}\end{equation}

Now consider the restricted vector field $h|\gamma'$, by parallel transport equation (\ref{clm33}) and induction on $q$
\begin{equation}\label{clm316}\begin{aligned}
\|\LIE^q_{h|\gamma'}(\oX-\oY)\|
&=\|h|\gamma'\LIE^{q-1}_{h|\gamma'}(\oX-\oY)-\LIE^{q-1}_{h|\gamma'}(\oX-\oY)\gamma'\|\\
&\leq\|\LIE^{q-1}_{h|\gamma'}h|\gamma'(\oX-\oY)\|+\|\LIE^{q-1}_{h|\gamma'}(\oX-\oY)\gamma'\|\\
&=\|\LIE^{q-1}_{h|\gamma'}(-h|\gamma'\cdot\Gamma\cdot(\oX-\oY))\|+\|\LIE^{q-1}_{h|\gamma'}(\oX-\oY)\gamma'\|\\
&<b^2(L^r+H^r)L^{2r}\alpha
\end{aligned}\end{equation}
by smoothness of $X$ and $Y$, and 
$$D\dis(\pi_J(x),x)\leq grad \dis(\pi_J(x),x)=\gamma'\dis(\pi_J(x),x)=1$$
one obtains
$$\|\LIE^q_{h|\Sigma}(X-Y)\|\leq\|\LIE^q_{h|\gamma'}(X-Y)\|$$
combine equation (\ref{clm31}), (\ref{clm315}) and (\ref{clm316}) results the $C^r$ size between $X$ and $Y$
\begin{equation}\label{clm317}\begin{aligned}
\|\LIE^q_{h_q}(X-Y)\|&=\|\LIE^q_{h_q|\oY+h_q|\Sigma(x)}(X-Y)\|\\
&\leq\|\LIE^q_{h_q|\oY}(X-Y)\|+\|\LIE^q_{h_q|\gamma'(x)}(X-Y)\|\\
&<\frac{\rho_0}{\epsilon^q}b^2L^{2r}(L^r+H^r)(\epsilon L^r+2)\alpha
\end{aligned}\end{equation}
for all $q=0,1,\dots,r$. Since $\rho_0, b, L, H$ are bounded constants independent of the almost periodic point $x_0$ chosen and $\epsilon>0$, $\alpha\to0$, the $C^r$ size between $X$ and $Y$ is arbitrarily small.
\end{proof}
As mentioned in the proof of uniqueness of the projection map in long flow box, if $\|X\|$ has no lower bound on some almost periodic orbit, the projection $\pi_J$ becomes nonunique. Similar construction as the proof of theorem \ref{clm1} could be applied to such an almost periodic orbit which creates a homoclinic orbit (the stable and unstable manifold of a stationary point which coincides). The existence of homoclinic orbit contradicts structural stability \cite{Rs1973}, since every system in the $C^r$ neighborhood of a stable system is stable. Therefore $X$ being structurally stable implies it has no almost periodic point close to equilibrium.
\begin{thm}\label{con1}
If there is an almost periodic point $x_0\in U(\omega)$ where $\|X(y)\|\to0$ for some $y=\phi_{t_0}(x_0)$, $0<t_0<T$, then there is $Y$ $C^r$-close to $X$ so that $Y(y)=0$ is a homoclinic point.
\end{thm}
\begin{proof}
 Similar to the construction in theorem \ref{clm1} a periodic orbit could be constructed which satisfies equation (\ref{clm03}) and (\ref{clm04}), i.e.
\[\begin{aligned}
\CD^q_{f'_t}f_0(x_0)&=\CD^q_{f'_t}f_T(x_0)\\
\|\overbrace{f'_t\dots f'_t}^q X-\overbrace{f'_t\dots f'_t}^q f'_t\|&<b^2(L^r+H^r)L^r\alpha
\end{aligned}\]
 Consider a reparameterization $\xi_t=f_{p(t)}$, so that
\[\begin{aligned}
\xi_t(x_{-1})=f_{p(t)}(x_{-1})&\quad\xi_{-t}(x_{1})=f_{-p(t)}(x_{1})\\
x_1=f_\tau(y)&\quad x_{-1}=f_{-\tau}(y)\\
p|0=identity&\quad and \quad 0\leq|p^{(q)}(t)|\leq1+\tau
\end{aligned}\] 
for all $q=0,\dots,r$ and $0<\tau<min\{t_0,T-t_0\}$. Denote $Y|\xi_t:=\xi'_t|\xi_t$, from inequality (\ref{clm04}) one obtains an estimation
\begin{equation}\label{hom01}\begin{aligned}
\|\overbrace{\xi'_t\dots\xi'_t}^q(X-Y)|\xi_t\|<(1+\tau)b^2(L^r+H^r)L^r\alpha
\end{aligned}\end{equation}
Similar to equation (\ref{clm11}), let $s(t)$ and $\overline{s}(t)$ be reparameterizations by arc length for $f_t$ and $\xi_t$ respectively,  since $\xi_t$ is already a reparameterization of $f_t$, $f_{s(t)}=\xi_{\overline{s}(t)}$. Set $\epsilon(x)$ as the radius of the long flow box for $\xi_t$ which satisfies
\begin{equation}
\epsilon(x)<\frac{1}{\|\CD^2_{\xi'_{\overline{s}(t)}}\xi_{\overline{s}(t)}\|}\quad where\quad \pi_J(x)=\xi_{\overline{s}(t)}
\end{equation}
For simplicity assume that there is a lower bound $L_{min}<X(x)$ for $x\notin U(y)$, one obtains an estimation similar to inequality (\ref{clm000})
\[\|\LIE^q_{h_q}(X-Y)|M\setminus U(y)\|
<\frac{\rho_0}{\epsilon(x)^q}b^2L^{2r}(L^r+H^r)(\epsilon(x) L^r+2)\alpha
\]
where $\epsilon(x)>0$. Assume that $\lim_{t\to\infty}f'_t=X(y)=0$ for $y\notin U(\omega)$, then clearly $\lim_{t\to\infty}\phi_t(x_0)=y\notin U(\omega)$, $x_0$ is not almost periodic, therefore $y\in U(\omega)$. 

By estimation (\ref{hom01}) $\lim_{x\to y}\epsilon(x)\to0$, $\pi_J(x)$ becomes nonunique for $x\in U(y)$. Consider scalar fields $\orho_i(\dis(\pi_{J_1}(x),x),\dis(\pi_{J_2}(x),x))$ with $J_1,  J_2\in U(y)$ 2 pieces of trajectories such that $\overline{J_1}\cap\overline{J_2}=y$, these scalar fields satisfy
\[\begin{aligned}
\orho_i(d_1,d_2)&=0\quad if\quad d_i\geq\epsilon\quad or\quad d_j=0\\
\orho_i(d_1,d_2)&=1\quad if\quad d_i=0\\
|\frac{\partial^q}{(\partial d_j)^q}\orho_i(d_1,d_2)
&+\frac{\partial^q}{(\partial d_i)^q}\orho_j(d_1,d_2)|<\rho_0\quad if\quad d_i,d_j<\frac{\epsilon}{2}\\
|\frac{\partial^q}{(\partial d_i)^q}\orho_i(d_1,d_2)|&<\frac{\rho_0}{\epsilon^q}\quad if\quad d_i<\epsilon
\end{aligned}\]
Then same construction as (\ref{clm31}) could be applied with $\rho$ replaced by $\orho$
\begin{equation}\label{hom20}
Y(x,t):=X(x,t)+\sum_i\orho_i(d_1,d_2,L|t-t_x|)(\oY_i(x)-\oX_i(x))
\end{equation}
where $Y_i=Y|\xi_t(x_i)$. Notice that
\[
\LIE^q_{h_q}\sum_i\orho_i(d_1,d_2,L|t-t_x|)(\oY_i(x)-\oX_i(x))|\lim_{t\to\infty}\Sigma(\xi_t(x_+))\cup\Sigma(\xi_{-t}(x_-))=0
\] 
since $X(y)=Y(y)=0$. One obtains the $C^r$ size similar to inequality (\ref{clm317})
\[\begin{aligned}
&\|\LIE^q_{h_q}(Y-X)\|=\|\LIE^q_{h_q}\sum_i\orho_i(d_1,d_2,L|t-t_x|)(\oY_i(x)-\oX_i(x))\|\\
&<max_{q'<q}|\LIE^{q'}_{h_{q'}}\sum_i\orho_i(d_1,d_2,L|t-t_x|)|\cdot max_{q'<q}\|\LIE^{q'}_{h_{q'}}\sum_i(\oY_i(x)-\oX_i(x))\|\\
&<\frac{\rho_0(\epsilon^r+1)}{\epsilon^r}b^2L^{2r}(L^r+H^r)(\epsilon L^r+2)\alpha
\end{aligned}\] 
which is arbitrarily small. Since $\lim_{t\to\infty}Y(\xi_t)=\lim_{t\to\infty}Y(\xi_{-t})=0$, $y$ is a homoclinic point of $Y$, hence a homoclinic point of $X$.
\end{proof}
As a corollary of the $C^r$ closing lemma, necessity of condition \romannumeral 1) is proved.
\begin{cor}\label{clm2}
If a system $X\in\SVF$ is $C^r$ structurally stable, $\Omega(X)=\overline{P(X)}$.
\end{cor}
\begin{proof}
The case $\NW(X)=P(X)$ is trivial, consider a nonwandering point $\omega\in\NW(X)\setminus P(X)$, let $x_i\in U(\omega)$ be a sequence of almost periodic points with period $T_i$ and $\lim_{i\to\infty}x_i=\omega$. Given any homeomorphism $h\in Hom(M)$, one obtains $lim_{i\to\infty}h(x_i)=h(\omega)$. By theorem \ref{clm1}, a field $Y\in\mathcal{U}(X)$ could be constructed such that there is a periodic orbit with period $T_i$ through every $h(x_i)$, which takes the form similar to (\ref{hom20})
\[
Y(x,t):=X(x,t)+\sum_i\orho_i(\dots,d_i,\dots,L|t-t_x|)(\oY_i(x)-\oX_i(x))
\]
If $\psi_t$ is the flow induced by $Y$, by definition of structural stability, there is a homeomorphism $h'$ such that $h'\circ\phi_t=\psi_t\circ h'$. Coincide $h'$ with $h$, one obtains
\[
h\circ\phi_{T_i}(x_i)=\psi_{T_i}(h(x_i))=h(x_i)
\]
i.e. $\phi_{T_i}(x_i)=x_i$ are periodic, thus every $\omega\in \overline{P(X)}$.
\end{proof}

\section{Proof of the Hyperbolicity Theorem}
Necessity of condition \romannumeral 2) for the stability conjecture is proved in theorem \ref{hyp1} (hyperbolicity of $\Omega(X)$). We take a very different approach from traditional ergodic ones, which assume no a priori structure of the (chaotic) system in question other than boundedness of the nonwandering set $\Omega(X)$. \par
Let $E^s(x)$, $E^u(x)$ be the stable and unstable subspace of $TM_x$ for flow $\phi_t$ induced by $X\in\SVF$. Utilizing corollary \ref{clm2} that $\Omega(X)=\overline{P(X)}$, for every $\omega\in\Omega(X)$ there is a sequence of periodic points $p_i\in P(X)$ $i=0,1,\dots$ with period $T_i$ and $\lim_{i\to\infty}p_i=\omega$. Recall that the hyperbolicity of periodic points for $C^1$ diffeomorphisms has been proved in \cite{Fr1971}, we extend this proof to the $C^r$ case.\par
 As mentioned in \cite{Ma1982}, an obstruction is the nonhyperbolic tendency of periodic points. We prove that $d\phi_{T_1(x)-T_0(x)}(p_i)|\Sigma$ has eigenvalues in the cross section bounded away from absolute value $1$.
\begin{lem}\label{hyp0}
If $X$ is structurally stable, for all periodic point $p\in P(X)$ there is a sub-period $0\leq T_0<T_1<\infty$ and a $\delta>0$ such that
\[\begin{aligned}\|d\phi_{T_1(p)}|E^s(p)\|\leq 1-\delta\\
\|d\phi_{-T_1(p)}|E^u(\phi_{T_1}(p))\|\leq 1-\delta
\end{aligned}\]
where $E^s\oplus E^u=T\Sigma$ is an invariant splitting on the cross-section.
\end{lem}
\begin{proof}
Since $\phi_{T_1}(p)$ returns arbitrarily $C^r$ close to $\phi_{T_0}(p)$, applying a perturbation to $X$ as (\ref{clm31}) yield a closed orbit. We want this closed orbit to have eigenvalues close to $1$ so as to perturb it to be nonhyperbolic. Denote $\phi_{T_0}(p)=p_0$, $T_1-T_0=T$. Suppose $L$ is a Lipschitz constant for $X$ $Y$ $T$ $\mP_\eta$ and the exponential map $exp$, consider estimations:
\begin{equation}\label{hyp01a}\begin{aligned}
&\|X(\phi_{T}(p_0)) dT(p_0,\phi)-\mP_\eta\circ X(\psi_T(p_0)) dT(p_0,\psi)\|\\
&\leq\frac{1}{2}\|X(\phi_{T}(p_0))+\mP_\eta\circ X(\psi_T(p_0))\|\cdot\|dT(p_0,\phi)-\mP_\eta\circ dT(p_0,\psi)\|\\
&+\frac{1}{2}\|X(\phi_{T}(p_0))-\mP_\eta\circ X(\psi_T(p_0))\|\cdot\|dT(p_0,\phi)+\mP_\eta\circ dT(p_0,\psi)\|
\leq2L\alpha
\end{aligned}\end{equation}
where $\mP_\eta$ is the parallel transport map between $TM_{\phi_t(p_0)}$ and $TM_{\psi_t(p_0)}$. As in lemma \ref{nonauto} $\dis(p_0,\phi_T(p_0))<\alpha$. Notice that $d_{p_0}\phi_t(p_0)=exp(\int^t_0dX(\phi_\tau(p_0))d\tau)$.
\begin{equation}\label{hyp01b}\begin{aligned}
\|\int^{T}_0dX\circ d_{p_0}\phi_t(p_0)dt-\mP_\eta\circ \int^{T}_0dX\circ d_{p_0}\psi_t(p_0)dt\|\leq TL^3\alpha
\end{aligned}\end{equation}
where $T,L<\infty$.
\begin{equation}\label{hyp01c}\begin{aligned}
&\|\mP_\eta\circ d\int^{T(p_0,\psi)}_0\rho(\dis(\psi_t(p_0),J),t)\cdot
\mT\circ (Y(\psi_t(p_0))-X(\psi_t(p_0)))dt\|\\
&\leq\|\rho(\dis(\psi_T(p_0),J),t)\cdot(Y(\psi_T(p_0))-X(\psi_T(p_0)))\cdot dT(p_0,\psi)\\
&+\int^T_0(Y(\psi_t(p_0))-X(\psi_t(p_0))\cdot\rho'\cdot\|grad\circ\dis(\cdot,\cdot)\|\cdot d_{p_0}\psi_t(p_0)dt\\
&+\int^T_0\rho\cdot d(Y(\psi_t(p_0))-X(\psi_t(p_0))\circ d_{p_0}\psi_t(p_0)dt\|
\leq L\alpha+T\alpha\frac{\rho_0}{\epsilon}L^2+L^2\alpha
\end{aligned}\end{equation}
where $\mT$ is the parallel transport map of (\ref{clm33}). Put together (\ref{hyp01a}), (\ref{hyp01b}) and (\ref{hyp01c}) yield the estimation between the linear maps:
\[
\|d\phi_{T(p_0,\phi)}(p_0)-\mP\circ d\psi_{T(p_0,\psi)}(p_0)\|\leq(2+TL^2+TL\frac{\rho_0}{\epsilon}+L+1)L\alpha
\]
which is arbitrarily small. Obviously their difference in eigenvalues restricted to $\Sigma$ is smaller than this bound. \par
Suppose $d\psi_{T(P_0,\psi)}|E^c=1-\delta$, consider a homeomorphism with parameter $A_T:\Sigma_T\to\Sigma_T$, then each $A_T$ preserves cross-section hence the return time $T(p_0,\psi)$, so that $A_{T(p_0,\psi)}\circ\psi_{T(p_0,\psi)}(p_0)\subset\Sigma_0$. Suppose that $A_T$ satisfies the following:
\begin{multline}
A_T(x)=\mI+\\
\sum_i\rho_i(\dis(x,\pi_1x),\dis(x,\pi_2x),\dots)\cdot\int^T_0\int_0^{\dis(x,\pi_i x)} d\mT_\eta \log dA_i\mT^{-1}_\eta\circ\gamma'_\eta d\eta dt
\end{multline}
where $\mI$ denotes the identity, $\mT_\eta$ $\gamma_\eta$ $\rho$ are as in theorem \ref{clm1}, $dA:T\Sigma\to T\Sigma$ is a linear map satisfying the following:
\[\begin{aligned}dA_T|E^c(\psi_T p_0)&=\frac{1}{\|d\psi_T|E^c(p_0)\|^2}d\psi_T|E^c(p_0)\\
 dA_T|T\Sigma\setminus E^c(\psi_T p_0)&=\mI
\end{aligned}\] 
now we can construct a vector field replacing $T(p_0,\psi)$ by $t$
\[\zeta'_t(x)=Z(\zeta_t x)=\frac{d}{dt}A_t(x)\circ\psi_t(x)
\]
whose return time is equal to $T(p_0,\psi)$. It is easy to verify it has eigenvalue $1$ at $p_0$
\[\|d\zeta_T(p_0,\psi)|E^c(p_0)\|=\|dA_T\circ d\psi_T|E^c(p_0)\|=1
\]
It remains to estimate the $C^r$ size between $Y$ and $Z$. Similar to (\ref{clm317})
\[\begin{aligned}
&\|\LIE^q_{h_q}(Y-Z)\|=
\|(\LIE^q_{h|\Sigma(x)}+\LIE^q_{h|Z})(Y-Z)\|\\
&<\|(\LIE^q_{h|\gamma'}+\LIE^q_{h|Z})(Y-Z)\|\\
&<\|(h|\gamma'\LIE^{q-1}_{h|\gamma'}-\LIE^{q-1}_{h|\gamma'}h|\gamma')(Y-Z)\|
+\|(h|Z\LIE^{q-1}_{h|Z}-\LIE^{q-1}_{h|Z}h|Z)(Y-Z)\|\\
&<\frac{\rho_0}{\epsilon^q}L^{q+1}\log\frac{1}{1-\delta}+L^q\frac{\delta}{1-\delta}
\end{aligned}\]
By assumption $\delta$ is arbitrarily small, hence $Z$ is arbitrarily $C^r$ close to $Y$ and to $X$, which has periodic orbit with eigenvalue $1$, so there exists a continuous band of periodic orbits. This contradict $C^r$ structural stability and prove the lemma.
\end{proof}
\begin{lem}\label{fin}
If $\dis(p,\omega)<\alpha$, then for any finite $\infty>T\geq t\geq0$,
\[\dis(\phi_tp,\phi_t\omega)\leq \mO(\alpha)
\] where $\mO$ is finite.
\end{lem}
\begin{proof}
Since $X\in\SVF$, $\phi_t(p)\neq\phi_t(\omega)$, and there is Lipschitz constant $L$ so that
\[\begin{aligned}
\frac{d}{dt}\dis(\phi_tp,\phi_t\omega)&\leq \|(\phi_tp-\mT\phi_t\omega)'\|=\|X(\phi_tp)-\mT X(\phi_t\omega)\|\\
&\leq L\dis(\phi_tp,\phi_t\omega)
\end{aligned}\]
Solving the differential inequality
$$\frac{d}{dt}\dis(\phi_tp,\phi_t\omega)\leq L\dis(\phi_tp,\phi_t\omega)$$
one obtains
$$\dis(\phi_tp,\phi_t\omega)\leq e^{Lt}\dis(p,\omega)\leq e^{LT}\dis(p,\omega)$$\end{proof}

The next theorem proves the necessity of condition \romannumeral 2).
\begin{thm}\label{hyp1}
If $X$ is $C^r$ structurally stable, there is a hyperbolic splitting $TM_\omega=E^s\oplus E^u$ for every $\omega\in\Omega(X)$, so that
$$\|d\phi_t|E^s(\omega)\|\leq C\lambda^t\quad t\geq0$$
$$\|d\phi_{-t}|E^u(\phi_t(\omega))\|\leq C\lambda^t\quad t\geq0$$
for some constants $C>0$, $0<\lambda<1$.
\end{thm}
\begin{proof}
By corollary \ref{clm2}, for each $\phi_{\sum_{j<i}T_j}\omega=\omega_i\in\NW(X)$, there is a sequence of periodic points $p_i$, so that $p_i\to\omega_i$. By lemma \ref{hyp0} these periodic points has finite sub-periods $T_i<\infty$. Since $d\phi_T$ is continuous, by lemma \ref{fin} there is a splitting $TM_{\omega}=E^s\oplus E^u$ so that $E^s(p_i)\to E^s(\omega_i)$ and $E^u(p_i)\to E^u(\omega_i)$. And there is a constant $\lambda<1-\delta$ so that
\[
\|d\phi_{T_i}|E^s(\omega_i)\|\leq\lambda\quad\|d\phi_{-T_i}|E^u(\phi_{T_i}(\omega_i))\|\leq\lambda
\]
by continuity of $d\phi_t$, there is a constant $C=sup\|d\phi_{t\bmod T}(\omega)\|$ and $d\phi_t$ satisfies
\[\|d\phi_{t}|E^s(\omega)\|\leq C\lambda^t\quad\|d\phi_{-t}|E^u(\phi_{t}(\omega))\|\leq C\lambda^t
\]
this completes the proof.
\end{proof}

\end{document}